\documentclass[10pt]{article}

\usepackage[final]{graphics}
\usepackage{graphicx}          	 
\usepackage{bm}                 
\usepackage{amsmath}             
\usepackage{amssymb}
\usepackage{amsfonts}             
\usepackage{verbatim}           
\usepackage{amsthm}             

\usepackage{mathtools}
\usepackage{relsize}

\numberwithin{equation}{section}	
\usepackage{hyperref}

\theoremstyle{plain}             

\newtheorem{theorem}{Theorem}[section]
\newtheorem{definition}[theorem]{Definition}
\newtheorem{lemma}[theorem]{Lemma}

\newtheorem{proposition}[theorem]{Proposition}
\theoremstyle{definition}
\newtheorem{example}[theorem]{Example}

\newtheorem{remark}[theorem]{Remark}

\input epsf

\def\calC{{{\cal C}}}
\def\sign{{\rm sign}}

\def\dsp{\displaystyle}
\def\eps{{\varepsilon}}

\def\RR{\mathbb R}

\begin{document}

\title{A distribution function from population genetics statistics using Stirling numbers of the first kind: Asymptotics, inversion  and numerical evaluation}

\author{
Swaine L. Chen\footnote{
Infectious Diseases Translational Research Programme and Department of Medicine, Division of Infectious Diseases, Yong Loo Lin School of Medicine, National University of Singapore, Singapore 119228, Singapore \&
Laboratory of Bacterial Genomics, Genome Institute of Singapore,
Singapore 138672, Singapore. Email: slchen@gis.a-star.edu.sg
}
\and
Nico M. Temme\footnote{IAA, 1825 BD 25, Alkmaar, The Netherlands. Former address: Centrum Wiskunde \& Informatica (CWI), Science Park 123, 1098 XG Amsterdam,  The Netherlands. Email: nico.temme@cwi.nl}
}


\maketitle
\begin{abstract}
Stirling numbers of the first kind are common in number theory and combinatorics; through Ewen’s sampling formula, these numbers enter into the calculation of several population genetics statistics, such as Fu's $F_s$. In previous papers we have considered an asymptotic estimator for a finite sum of Stirling numbers, which enables rapid and accurate calculation of Fu's $F_s$. These sums can also be viewed as a cumulative distribution function; this formulation leads directly to an inversion problem, where, given a value for Fu's $F_s$, the goal is to solve for one of the input parameters. We solve this inversion using Newton iteration for small parameters. For large parameters we need to  extend the earlier obtained asymptotic results to handle the inversion problem asymptotically. Numerical experiments are given to show the efficiency of both solving the inversion problem and  the expanded estimator for the statistical quantities.
\end{abstract}


\section{Introduction}\label{sec:intro}
In  recent papers \cite{Chen:2019:ISN}  and \cite{Chen:2020:FMA}  we have discussed the sum
\begin{equation}\label{eq:intro01}
S^\prime_{n,m}(\theta)=\frac{1}{(\theta)_n}\sum_{k= m}^n(-1)^{n-k}S_n^{(k)}\theta^k,\quad \theta>0,
\end{equation}
where 
 $S_n^{(k)}$ are the  Stirling numbers of the first kind defined by
\begin{equation}\label{eq:intro02}
(\theta)_n=\sum_{k= 0}^n(-1)^{n-k}S_n^{(k)}\theta^k,
\end{equation}
and $(\theta)_n$ is the Pochhammer symbol, defined by
\begin{equation}\label{eq:intro03}
(\theta)_0=1,\quad (\theta)_n=\theta(\theta+1)\cdots(\theta+n-1)=\frac{\Gamma(\theta+n)}{\Gamma(\theta)}.
\end{equation}

The quantity $S^\prime_{n,m}(\theta)$ (and related quantities) is used in the calculation of several population genetics statistics. One such statistic is Fu's $F_s$,
\begin{equation}\label{eq:intro04}
F_s=\ln\frac{S^\prime_{n,m}(\theta)}{1-S^\prime_{n,m}(\theta)},
\end{equation}
which was shown to be capable of identifying the genetic changes responsible for the increased fitness of a recently expanded clone of Campylobacter jejuni that is causing an epidemic of abortion in livestock \cite{pmid27601641}. We used asymptotic approximations of the Stirling numbers derived in  \cite{Temme:1993:AES} to compute $S^\prime_{n,m}(\theta)$ \cite{Chen:2019:ISN}. Subsequently, we transformed the sum into a contour integral in the complex plane, and we gave a first-order approximation of this integral  for large $n$, with $0< m< n$ and $\theta>0$ \cite{Chen:2020:FMA}.

Note that, because of \eqref{eq:intro02} the sum \eqref{eq:intro01} satisfies $0\le S^\prime_{n,m}(\theta)\le1$; see Figure~\ref{fig:fig01}. In fact $S^\prime_{n,m}(\theta)$ can be viewed as a cumulative distribution function  frequently used in the derivation of several population genetics statistics, which in turn are useful for testing evolutionary hypotheses directly from DNA sequences. 

In the earlier paper \cite{Chen:2020:FMA}  we have derived a new integral representation of  $S^\prime_{n,m}(\theta)$ and we have given a first-order asymptotic approximation in terms of an incomplete beta function. With these results the algorithm given  in a first attempt \cite{Chen:2019:ISN} could be considerably improved in efficiency and speed.

In  the present paper the main interest is the inversion problem to find $\theta$ from the equation  $S^\prime_{n,m}(\theta)=s$ with given $s\in(0,1)$, $n$ and $m$. For small or intermediate values of $n$ we use a Newton iteration scheme, whereas for large $n$  we derive more details on the earlier derived asymptotic representation of $S^\prime_{n,m}(\theta)$  to develop an asymptotic expansion of the wanted $\theta$. Numerical experiments are given to show the efficiency of the asymptotic expansion, of the Newton iterations, and the asymptotic inversion problem.

 \begin{figure}[tb]
\vspace*{0.8cm}
\begin{center}
\begin{minipage}{5cm}
    \includegraphics[width=4.5cm]{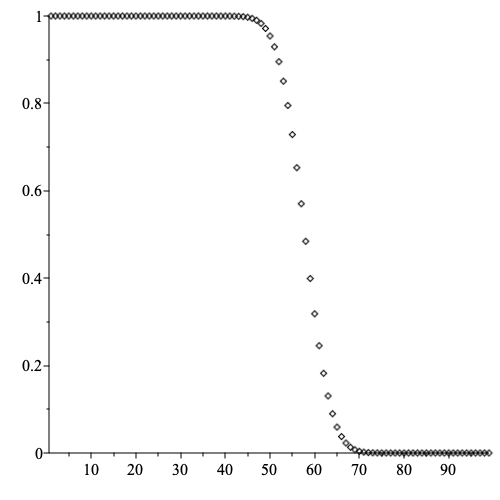} 
\end{minipage}
\hspace*{1cm}
\begin{minipage}{5cm}
    \includegraphics[width=4.5cm]{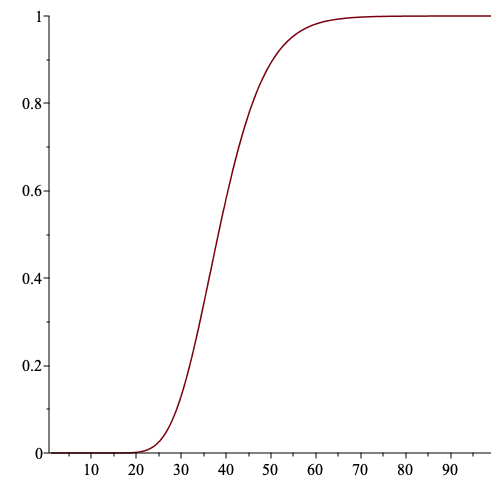} 
\end{minipage}
\end{center}
\caption{\small
{\bf Left:} The function $S^\prime_{n,m}(\theta)$  for $m=0,1,2,\ldots, 100$; $\theta=50$ and $n=100$.
{\bf Right:} The function $S^\prime_{n,m}(\theta)$  for $\theta\in[0,100]$; $m=50$ and $n=100$.
}
\label{fig:fig01}
\end{figure}

\section{A few details on the Stirling numbers}\label{sec:details}
For a concise overview of properties, of these Stirling numbers, with a summary of their uniform approximations, see \cite[\S11.3]{Gil:2007:NSF}.

In Figure~\ref{fig:fig01} we show two graphs of $S^\prime_{n,m}(\theta)$, on the left left a point plot for $0\le m\le n$, with fixed $\theta=50$ and $n=100$, and on the  right a smooth sigmoid curve for $0\le\theta\le n$ with fixed $m=50$ and $n=100$.

A key representations for the considered asymptotic problem is the Cauchy-type integral in the complex plane
\begin{equation}\label{eq:details01}
(-1)^{n-k}S_n^{(k)}=\frac{1}{2\pi i}\int_{\calC}\frac{(z)_n}{z^{k+1}}\,dz,
\end{equation}
which follows from \eqref{eq:intro02}. Here $\calC$ is a contour around the origin. 

Special values are
\begin{equation}\label{eq:details02}
S_n^{(n)}=1 \  (n\ge0), \quad S_n^{(0)}=0\ (n\ge1), \quad S_n^{(1)}=(-1)^{n-1}(n-1)! \ (n\ge1),
\end{equation}
and there is a recurrence relation:
\begin{equation}\label{eq:details03}
S_{n+1}^{(k)}=S_{n}^{(k-1)}-nS_{n}^{(k)}.
\end{equation}
For the sums  $S^\prime_{n,m}(\theta)$ we have a new similar result.

\begin{theorem}\label{thm:stel01}
The sums $S^\prime_{n,m}(\theta)$ satisfy for $n=2,3,4,\ldots$
 the recursion
 \begin{equation}\label{eq:details04}
\begin{array}{r@{\,}c@{\,}l}
(\theta+n)S^\prime_{n+1,m}(\theta)&=&nS^\prime_{n,m}(\theta)+\theta S^\prime_{n,m-1}(\theta),\quad 1\le m\le n,\\[8pt]
S^\prime_{n+1,n+1}(\theta)&=&\dsp{\frac{\theta^{n+1}}{(\theta)_{n+1}},}
\end{array}
\end{equation}
with initial values
\begin{equation}\label{eq:details05}
\begin{array}{r@{\,}c@{\,}l}
S^\prime_{0,0}(\theta)&=&1,\quad S^\prime_{1,0}(\theta)=1,\quad S^\prime_{1,1}(\theta)=1,\\[8pt]
S^\prime_{2,0}(\theta)&=&1,\quad S^\prime_{2,1}(\theta)=1,\quad \dsp{S^\prime_{2,2}(\theta)=\frac{\theta}{\theta+1}.}\\[8pt]
\end{array}
\end{equation}
\end{theorem}
\begin{proof}
The proof simply follows from using the relation in \eqref{eq:details03}. We also  need $S_n^{(n)}=1$ and $(\theta)_{n+1}=(\theta+n)(\theta)_n$.
\end{proof}

Observe that there are no Stirling numbers in the recursion in \eqref{eq:details04}, apart from those needed to compute the starting values in \eqref{eq:details05}. This gives a very simple method to compute $S^\prime_{n,m}(\theta)$ for small or intermediate values of $n$. For large values of $n$ we prefer asymptotic representations.

The  Stirling numbers are very large for large values of  $n$ and $m\ll n$, see the value of $S_n^{(1)}$ in \eqref{eq:details02}. This makes straightforward evaluation of the sum in \eqref{eq:intro01} sensitive to overflow. This problem does not happen for the recursion  in \eqref{eq:details04}, because $0\le S^\prime_{n,m}(\theta)\le1$.

In the initial values we see that $S^\prime_{n,0}(\theta)=S^\prime_{n,1}(\theta)=1$ for $n=0,1$. More generally this follows from $S^{(0)}_n=0$ if $n\ge1$.

In the computation of cumulative distribution functions, like the classical gamma and beta cases, it is essential to consider the complementary relations. In the present case we use the complementary sum
\begin{equation}\label{eq:details06}
T^\prime_{n,m}(\theta)=1-S^\prime_{n,m}(\theta)=\frac{1}{(\theta)_n}\sum_{k= 0}^{m-1}(-1)^{n-k}S_n^{(k)}\theta^k.
\end{equation}
To avoid numerical cancellation when using the complementary relation, we should compute first the {\em primary} function, that is, $\min(S^\prime_{n,m}(\theta),T^\prime_{n,m}(\theta))$, and the other one from the complementary relation. The functions $T^\prime_{n,m}(\theta)$ satisfy the same recursion as $S^\prime_{n,m}(\theta)$ in \eqref{eq:details04}, of course with different starting values.

As mentioned above, Fu's $F_s$ in \eqref{eq:intro04} is of interest for population genetics applications \cite{pmid9335623}. From \eqref{eq:details06}, we also have a complementary equation for Fu's $F_s$,
\begin{equation}\label{eq:details07}
F_s=\ln\frac{1-T^\prime_{n,m}(\theta)}{T^\prime_{n,m}(\theta)}.
\end{equation}
Fu's $F_s$ ranges from $-\infty$ to $+\infty$ as $\theta$ runs through the interval $(0,\infty)$ and it vanishes when $\theta$ equals its {\em transition value} $\theta_t$ for which value  we have 
\begin{equation}\label{eq:details08}
S^\prime_{n,m}(\theta_t)=T^\prime_{n,m}(\theta_t)=\tfrac12.
\end{equation}

Both representations of $F_s$ are needed in numerical computations, because when $S^\prime_{n,m}(\theta)$ is close to 1, the form with $T^\prime_{n,m}(\theta)$ in \eqref{eq:details07} gives a more reliable computational representation.

\section{Summary of earlier results}\label{sec:earlier}
Because it is more convenient to work with 
$S^\prime_{n+1,m+1}(\theta)$ we proceed with
\begin{equation}\label{eq:summ01}
\begin{array}{r@{\,}c@{\,}l}
S_{n+1,m+1}(\theta)&=&\dsp{\frac{1}{(\theta+1)_{n}}\sum_{k= m}^{n} (-1)^{n-k}S_{n+1}^{(k+1)}\theta^k,}\\[8pt]
T_{n+1,m+1}(\theta)&=&\dsp{\frac{1}{(\theta+1)_{n}}\sum_{k= 0}^{m-1} (-1)^{n-k}S_{n+1}^{(k+1)}\theta^k.}
\end{array}
\end{equation}
The following result of our paper \cite{Chen:2020:FMA} is crucial for deriving asymptotic expansions of the statistical quantities.
\begin{theorem}\label{thm:stel02}
 Let $\calC_\rho$ be  a circle at the origin of the complex plane with radius $\rho>0$. Then $S_{n+1,m+1}^\prime(\theta)$ and $T_{n+1,m+1}^\prime(\theta)$ have representations as contour integrals
\begin{equation}  \label{eq:summ02}
\begin{array}{r@{\,}c@{\,}l}
S_{n+1,m+1}^\prime(\theta)&=&\dsp{\frac{\theta^{m}}{(\theta+1)_n}\frac{1}{2\pi i}\int_{\calC_\rho} \frac{(z+1)_n}{z^{m}}\,\frac{dz}{z-\theta},\quad \rho > \theta,}\\[8pt]
  T_{n+1,m+1}^\prime(\theta)&=&\dsp{\frac{\theta^{m}}{(\theta+1)_n}\frac{1}{2\pi i}\int_{\calC_\rho} \frac{(z+1)_n}{z^{m}}\,\frac{dz}{\theta-z},\quad \rho < \theta.}
\end{array}
\end{equation}
Here, $n$ and $m$ are positive integers, $0\le m\le n$, and $\theta$ is a real positive number.
The symbol $(\alpha)_n$ denotes the Pochhammer symbol introduced in \eqref{eq:intro03}. 
 \end{theorem}

The main asymptotic results follow from representations given in  \cite{Chen:2020:FMA} and are  summarised in the  next theorem.
\begin{theorem}\label{thm:stel03}
$S_{n+1,m+1}^\prime(\theta)$ and $T_{n+1,m+1}^\prime(\theta)$ have the representations 
\begin{equation}
\begin{aligned}
 & S_{n+1,m+1}^\prime(\theta)=I_x(m, n-m+1)+R_{n+1,m+1}^\prime(\theta), \quad & x= {\frac{\tau}{1+\tau}},\\
& T_{n+1,m+1}^\prime(\theta)=I_{1-x}(n-m+1,m)-R_{n+1,m+1}^\prime(\theta), \quad & 1-x= {\frac{1}{1+\tau}},
  \label{eq:summ03}
\end{aligned}
\end{equation}
where $I_x(p,q)$ is the incomplete beta function defined by
\begin{equation}\label{eq:summ04}
I_x(p,q)=\frac{1}{B(p,q)}\int_0^x t^{p-1}(1-t)^{q-1}\,dt,
\end{equation}
with
\begin{equation}\label{eq:summ05}
0 < x < 1,\quad p >0, \quad q> 0, \quad B(p,q)=\frac{\Gamma(p)\Gamma(q)}{\Gamma(p+q)}.
\end{equation}
The term $R_{n+1,m+1}^\prime(\theta)$ can be expanded in negative powers of the large parameter $n$, as will be explained in later sections. For the relation between $\tau$ and $\theta$ we refer to Definition~\ref{def:def01}. \end{theorem}

Observe that in the representations given in this theorem the complementary relation in \eqref{eq:details06} is preserved because of the complementary property of the incomplete beta function:
\begin{equation}\label{eq:summ06}
I_x(p,q)=1-I_{1-x}(q,p).
\end{equation}

\subsection{The asymptotic approach}\label{sec:approach}
The representations given  in Theorem~\ref{thm:stel03} are obtained in  \cite{Chen:2020:FMA} by using the saddle point method. The first step is the introduction of a phase function $\phi(z)$. We write
\begin{equation}\label{eq:summ07}
S_{n+1,m+1}^\prime(\theta)=\frac{e^{-\phi(\theta)}}{2\pi i}\int_{\calC_\rho} e^{\phi(z)}
\,\frac{dz}{z-\theta},\quad \rho>\theta,
\end{equation}
where  ${\calC_\rho}$ is a circle as in Theorem~\ref{thm:stel02} and 
\begin{equation}\label{eq:summ08}
\begin{aligned}
\phi(z)&=\ln\left( (z+1)_n\right)-m\ln z =\sum_{k=0}^{n-1}\ln(z+1+k)-m\ln z\\
&=\ln\Gamma(z+1+n)-\ln\Gamma(z+1)-m\ln z.
\end{aligned}
\end{equation}
For positive values of $z$, we have the limiting forms
\begin{equation}\label{eq:summ09}
\phi(z)\sim -m\ln z, \quad z \to 0;\quad \quad \phi(z)\sim n\ln (z+1),\quad z\to\infty,
\end{equation}
where the second estimate comes from $\Gamma(z+1+n)/\Gamma(z+1)\sim(z+1)^n$. In addition, there is one positive minimum $z_0$ of $\phi(z)$. A proof is given in  \cite{Temme:1993:AES}.

The function 
\begin{equation}\label{eq:summ10}
\chi(t)=n\ln(t+1)-m\ln t, \quad t>0,
\end{equation}
has the same limiting behaviour as $\phi(z)$ at $t=0$ and as $t\to\infty$, and it has one positive minimum $t_0=m/(n-m)$. In fact, these functions behave quite similar for positive values of their arguments, and we have the following lemma.
\begin{lemma}\label{lem:lem01}
Consider for positive $z$ and $t$ the equation
\begin{equation}\label{eq:summ11}
\phi(z)-\phi(z_0)=\chi(t)-\chi(t_0).
\end{equation}
Then there is a one-to-one relation between $z$ and $t$ when we use the following condition: $\sign(z-z_0)=\sign(t-t_0)$.
\end{lemma}

\begin{proof}
In Figure~1 of \cite{Chen:2020:FMA} we have drawn graphs of both functions for $m=38$ and $n=100$. Both derivaties of the non-negative convex functions have a unique positive zero  $z_0$ and $t_0$ and their convex curves touch the positive real axis at $z_0$ and $t_0$. The sign condition  for the relation in \eqref{eq:summ11} means that the left branches of the curves correspond with functions values for $z\in(0,z_0]$ and $t\in(0,t_0]$, and the right branches with values for $z\in[z_0,\infty)$ and $t\in[t_0,\infty)$. Clearly, we can uniquely determine $z(t)$ and $t(z)$ for positive values of these parameters.

\end{proof}

Before we start with deriving the asymptotic representations given in Theorem~\ref{thm:stel03} we define the following special points used in this paper.

\begin{definition}\label{def:def01}{\bf Special points}

\begin{enumerate}
\item
The point $z_0$, the positive  minimum of $\phi(z)$ and the positive solution of the equation $\phi^\prime(z)=0$, where
\begin{equation}  \label{eq:summ12}
\begin{aligned}
  \phi^\prime(z)&=\sum_{k=0}^{n-1}\frac{1}{z+1+k}-\frac{m}{z}\\
  &=\psi(z+n+1)-\psi(z+1)-\frac{m}{ z}, \quad \psi(z)=\frac{\Gamma^\prime(z)}{\Gamma(z)},
\end{aligned}
\end{equation}
is called the {\em saddle point} of the integral in \eqref{eq:summ07}.

\item
The value $\theta_t$ for which $S^\prime_{n,m}(\theta_t)=T^\prime_{n,m}(\theta_t)=\frac12$ is called  the {\em transition point} of  $S^\prime_{n,m}(\theta)$ and $T^\prime_{n,m}(\theta)$. 
\item
The positive value of $t$ that satisfies the relation in \eqref{eq:summ11} when $z$ is replaced by $\theta$ is called $\tau$. That is, when we write the general solution of \eqref{eq:summ11}  as $t(z)$, then $\tau=t(\theta)$. Also, $\tau$ is the positive solution of
\begin{equation}\label{eq:summ13}
\phi(\theta)-\phi(z_0)=\chi(\tau)-\chi(t_0),\quad \sign(\theta-z_0)=\sign(\tau-t_0).
\end{equation}
\end{enumerate}
\end{definition}

The relation in \eqref{eq:summ11} will be used as a transformation of variables (which is also used in \cite{Temme:1993:AES}) and  \eqref{eq:summ07} becomes
\begin{equation}\label{eq:summ14}
\begin{aligned}S_{n+1,m+1}^\prime(\theta)&=
\frac{e^{-\chi(\tau)}}{2\pi i}\int_{{\cal C}_\sigma} \frac{(t+1)^n}{t^{m}}f(t)\,dt,\\
f(t)&=\frac{1}{z-\theta}\frac{dz}{dt},\quad \frac{dz}{dt}=\frac{\chi^\prime(t)}{\phi^\prime(z)},
\quad \chi^{\prime}(t)=(n-m)\frac{t-t_0}{t(1+t)},
\end{aligned}
\end{equation}
with $t_0=m/(n-m)$, where we have used  the relation for $\chi(\tau)$ in \eqref{eq:summ13}. The value $z=\theta$ (a pole of the integrand in  \eqref{eq:summ07}) corresponds with  $t=\tau$ (see Definition~\ref{def:def01}), and this  means that the function $f(t)$ will have a pole at $t=\tau$.

The main asymptotic result of \cite{Chen:2020:FMA} is given in the following theorem.

\begin{theorem}\label{thm:stel04}
Let the function $g(t)$ be defined by 
\begin{equation}\label{eq:summ15}
g(t)=f(t)-\frac{1}{t-\tau},
\end{equation}
where $f(t)$ is defined in \eqref{eq:summ14}. Then the function $R_{n+1,m+1}^\prime(\theta)$ of the  representations given in \eqref{eq:summ03}, has the integral representation
\begin{equation}\label{eq:summ16}
R_{n+1,m+1}^\prime(\theta)=
\frac{e^{-\chi(\tau)}}{2\pi i}\int_{{\cal C}_\sigma} \frac{(t+1)^n}{t^{m}}g(t)\,dt,
\end{equation}
where $\calC_\sigma$ is a contour around the origin and inside the domain where $g(t)$  is analytic. 
A first-order approximation is given by
\begin{equation}\label{eq:summ17}
R_{n+1,m+1}^\prime(\theta)\sim e^{-\chi(\tau)}\binom{n}{m-1}\ g(t_0),\quad t_0=\frac{m}{n-m},
\end{equation}
where
\begin{equation}\label{eq:summ18}
g(t_0)=f(t_0)-\frac{1}{t_0-\tau},\quad f(t_0)=\frac{1}{z_0-\theta}\sqrt{\frac{\chi^{(2)}(t_0)}{\phi^{(2)}(z_0)}}.
\end{equation}
\end{theorem}

The incomplete beta function (see  \eqref{eq:summ04}) is used with the representation
\begin{equation}\label{eq:summ19}
I_{\frac{\tau}{1+\tau}}(m, n-m+1)=\frac{e^{-\chi(\tau)}}{2\pi i}\int_{{\cal C}_\sigma} \frac{(t+1)^n}{t^{m}}\,\frac{dt}{t-\tau}.
\end{equation}
This function has the representation (see \cite[\S8.17(i)]{Paris:2010:INC}) 
\begin{equation}\label{eq:summ20}
I_\frac{\tau}{1+\tau}(m,n-m+1)=(1+\tau)^{-n}\sum_{j=m}^n\binom{n}{j}\tau^j,
\end{equation}
and from the complementary relation in \eqref{eq:summ06} it follows that $I_\frac{1}{1+\tau}(n-m+1,m)$ used in \eqref{eq:summ03} has the expansion
\begin{equation}\label{eq:summ21}
I_\frac{1}{1+\tau}(n-m+1,m)=(1+\tau)^{-n}\sum_{j=0}^{m-1}\binom{n}{j}\tau^j.
\end{equation}

\begin{remark}\label{rem:remark01}
The role of the transition point $\theta_t$  in connection with Fu's $F_s$ is explained before   \eqref{eq:details08}. The transition value $\theta_t$ can be obtained by using the inversion methods of \S\ref{sec:invert}. On the other hand, the main term in the first line of \eqref{eq:summ03} is the incomplete beta function, which function has \cite{Gil:2017:EAF}  the transition point $x_t$  close to $x=p/(p+q)$, which gives $\tau=m/(n+m+1)$. This point is close to  $t_0=m/(n+m)$, the zero of $\chi^\prime(t)$, see \eqref{eq:summ14} and saddle point of the integral in \eqref{eq:summ14}, which corresponds to $z_0$, the saddle point of integral in \eqref{eq:summ07}.
We conclude  that the saddle point $z_0$ is a good approximation of the transition value $\theta_t$. 
\end{remark}

\section{The inversion problem}\label{sec:invert}
We consider the following inversion problem: let $m$ and $n$ be given, together with a value $s\in(0,1)$.
Then find $\theta$ such that 
\begin{equation}\label{eq:inv01}
S_{n,m}^\prime(\theta)=s.
\end{equation}
Since $S_{n,m}^\prime(0)=0$ and $S_{n,m}^\prime(\theta)\to 1$ as $\theta\to\infty$, and $S_{n,m}^\prime(\theta)$ is an increasing function of $\theta$ (see also Figure~\ref{fig:fig01} ({\bf Right}), there is a unique solution $\theta$ of this problem.

We consider two approaches to solve this problem, in the first one we use Newton iteration and the other one  is especially useful when the parameters $n$ and $m$ are large enough to use asymptotic approximations.

In both methods we use a starting value $\theta_0$ that follows  from the value $x$ that solves the reduced equation
\begin{equation}\label{eq:inv02}
I_{x}(p, q)=s, \quad p=m, \quad q=n-m+1, \quad x=\frac{\tau}{1+\tau},
\end{equation}
where $I_x(p,q)$ is the incomplete beta function used in \eqref{eq:summ03}. With this value $x$ we compute $\tau=x/(1-x)$ and the initial value $\theta_0$ in the Newton method  then follows from the relation between $\tau$ and $\theta$  as explained in  Definition~\ref{def:def01}. 
We use the reduced equation  because we consider the incomplete beta function as the main asymptotic approximant of $S_{n+1,m+1}^\prime(\theta)$. Also for small values of $n$ and $m$ it gives a useful initial value $\theta_0$.

The inversion of the incomplete beta function is extensively considered in the literature.  For an approach for large variable $p$ and $q$, see 
\cite{Temme:1992:AIB} while in \cite{Gil:2017:EAF}  a fourth order fixed point method and several other approaches are discussed.
For an overview  of  the inversion of other classical cumulative distribution functions, we refer to \cite[Chapter~42]{Temme:2015:AMI}.

\begin{remark}\label{rem:remark02}
The inversion $S_{n,m}^\prime(\theta)=s$  can be replaced by the equation  for the complementary function: $T_{n,m}^\prime(\theta)=1-s$, which is relevant when  $s\sim1$, and even more relevant when $s=1-\sigma$, when $\sigma$ is known in detail as a small positive number. 
\end{remark}

\begin{remark}\label{rem:remark03}
The inversion of Fu's $F_s$ (see\eqref{eq:intro04}), that is, solving the equation $F_s=f$, $f\in\RR$, follows immediately from our methods for solving $S_{n,m}^\prime(\theta)=s$. The equation   $F_s=f$ is equivalent with solving
\begin{equation}\label{eq:inv03}
S_{n,m}^\prime(\theta)=\frac{e^f}{1+e^f},\quad T_{n,m}^\prime(\theta)=\frac{1}{1+e^f}.
\end{equation}
When $f$ is a large positive number, it is very relevant to solve the second equation.
\end{remark}

\subsection{The iterative  inversion method}\label{sec:invmeth1}
When solving the equation in \eqref{eq:inv01} with the Newton iterative method we compute a sequence of vales $\theta_j$, $j=0,1,\ldots$, from the scheme
\begin{equation}\label{eq:inv04}
\theta_{j+1}=\theta_j-\frac{f(\theta_j)-s}{f^\prime(\theta_j)},\quad j=0,1,2,\ldots, \quad  f(\theta)=S_{n,m}^\prime(\theta).
\end{equation}
The starting value $\theta_0$ is obtained from the reduced equation in \eqref{eq:inv02}. The derivative of $S_{n,m}^\prime(\theta)$ follows from  the following lemma.

\begin{theorem}\label{thm:stel05}
\begin{equation}\label{eq:inv05}
\frac{d}{d\theta} S_{n,m}^\prime(\theta)=-S_{n,m}^\prime(\theta)\sum_{k=0}^{n-1}\frac{1}{k+\theta}+\widehat{S}_{n,m}^\prime(\theta),
\end{equation}
where
\begin{equation}\label{eq:inv06}
\widehat{S}_{n,m}^\prime(\theta)=
\frac{1}{(\theta)_n}\sum_{k= m}^n(-1)^{n-k}kS_n^{(k)}\theta^{k-1}.
\end{equation}
These functions satisfy the recurrence relation
\begin{equation}\label{eq:inv07}
(\theta+n)\widehat{S}^\prime_{n+1,m}(\theta)=n\widehat{S}^\prime_{n,m}(\theta)+\theta \widehat{S}^\prime_{n,m-1}(\theta)+S^\prime_{n,m-1}(\theta),\quad 1\le m \le n. 
\end{equation}
\end{theorem}
\begin{proof}
First we have
\begin{equation}\label{eq:inv08}
\frac{d}{d\theta}\frac {1}{(\theta)_n}=\frac{d}{d\theta}\frac {\Gamma(\theta)}{\Gamma(\theta+n)}=
\frac {\Gamma^\prime(\theta)}{\Gamma(\theta+n)}-\frac {\Gamma(\theta)\Gamma^\prime(\theta+n)}{\Gamma^2(\theta+n)}.
\end{equation}
Next we use the $\psi$-function, defied by $\psi(z)=\Gamma^\prime(z)/\Gamma(z)$, which has the recursive property (which is also used in \eqref{eq:summ12})
\begin{equation}\label{eq:inv09}
\psi(z+n)=\psi(z)+\sum_{k=0}^{n-1}\frac{1}{k+z},\quad n=1,2,3,\ldots.
\end{equation}
This recursion easily follows from the fundamental property of the gamma function $\Gamma(z+1)=z\Gamma(z)$. 
We find
\begin{equation}\label{eq:inv10}
\frac{d}{d\theta}\frac {1}{(\theta)_n}=\left(\psi(\theta)-\psi(\theta+n)\right)\frac {1}{(\theta)_n}.
\end{equation}
Combining these results we find the relation in \eqref{eq:inv05}. The proof of the recurrence relation in \eqref{eq:inv07} follows from the recurrence relation of the Stirling numbers in \eqref{eq:details03},  just as in the proof of Theorem~\ref{thm:stel01}.
\end{proof}

A few first values of $\widehat{S}^\prime_{n,m}(\theta)$ are
\begin{equation}\label{eq:inv11}
\begin{array}{r@{\,}c@{\,}l}
\widehat{S}^\prime_{0,0}(\theta)&=&0,\quad \widehat{S}^\prime_{1,0}=\dsp{\frac{1}{\theta}},\quad \widehat{S}^\prime_{1,1}=\dsp{\frac{1}{\theta}},\\[8pt]
\widehat{S}^\prime_{2,0}(\theta)&=&\dsp{\frac{2\theta+1}{\theta(\theta+1)}},\quad \widehat{S}^\prime_{2,1}=\dsp{\frac{2\theta+1}{\theta(\theta+1)}},\quad \dsp{\widehat{S}^\prime_{2,2}=\frac{2}{\theta+1}.}
\end{array}
\end{equation}

In \S\ref{sec:numres} we give numerical examples of the Newton iteration scheme.

\subsection{The asymptotic  inversion method}\label{sec:invmeth2}

We consider the inversion of the full equation \eqref{eq:inv01} with the  representation  of $S_{n+1,m+1}^\prime(\theta)$ as given in  \eqref{eq:summ03}. We concentrate on finding  $\tau$ and with this information we compute $\theta$; see Definition~\ref{def:def01}. We propose the following 

\begin{proposition}\label{prop:prop01}
Let $x$ be the solution of the reduced equation in \eqref{eq:inv02}, with corresponding $\tau$ value $\tau_0=x/(1-x)$. Then we will construct an expansion of the wanted value $\tau$ of the form
\begin{equation}\label{eq:inv12}
\tau=\tau_0+\eps,\quad \eps\sim \frac{\tau_1}{\nu}+\frac{\tau_2}{\nu^2}+\ldots, \quad \nu=n-m,
\end{equation}
with 
\begin{equation}\label{eq:inv13}
\tau_1 =\frac{\tau_0(\tau_0+1)}{\tau_0-t_0}\ln\bigl((t_0-\tau_0)f(t_0)\bigr),
\end{equation}
where  $f(t_0)$ is given in \eqref{eq:summ18}, and 
\begin{equation}\label{eq:inv14}
\begin{array}{r@{\,}c@{\,}l}
e^{-\xi}\tau_2&=&\dsp{\tau_1\frac{e^{-\xi}-1}{\xi}+
\frac{(2\tau_0+1)\tau_1^2}{\tau_0(\tau_0+1)}\frac{e^{-\xi}-1+\xi}{\xi^2}\ +}\\[8pt]
&&\dsp{\rho^\prime(\tau_0)\tau_1^3\frac{e^{-\xi}-1+\xi-\frac12\xi^2}{\xi^3}\ -}\\[8pt]
&&
\dsp{\tau_0(\tau_0+1)\left(G_1(t_0)+\tau_1G_0^\prime(t_0) +\tfrac{1}{2}\rho^\prime(\tau_0)\tau_1^2 G_0(t_0)\right).}
\end{array}
\end{equation}
The coefficients $G_k(t)$ are defined in Theorem~\ref{thm:stel06} (see also \S\ref{sec:deriv}) and
\begin{equation}\label{eq:inv15}
\rho(\tau)=\frac{t_0-\tau}{\tau(1+\tau)},\quad \xi=\tau_1\rho(\tau_0).
\end{equation}
\end{proposition}

To obtain these coefficients $\tau_j$ we have used a perturbation method that starts with  writing $S_{n+1,m+1}^\prime(\theta)$ of Theorem~\ref{thm:stel03} in the form 
\begin{equation}\label{eq:inv16}
I_{\frac{\tau_0+\eps}{1+\tau_0+\eps}}(p,q)+e^{-\chi(\tau_0+\eps)}\binom{n}{m-1}S(\tau_0+\eps)=s,
\end{equation}
where $S(\tau)$ is the function with expansion (see  \eqref{eq:deriv01} and \eqref{eq:summ03})
\begin{equation}\label{eq:inv17}
S(\tau)\sim\sum_{k=0}^\infty \frac{G_k(t_0)}{\nu^k}.
\end{equation}

The idea is to use  the expansion of $\eps$ given in \eqref{eq:inv12}  in \eqref{eq:inv16}, and expand the relevant terms in  negative powers of $\nu$. The coefficients of the same powers of $\nu$ in this collection should vanish. This yields  equations for the coefficients $\tau_j$.
Because we already calculated $\tau_0$ such  that $I_{\frac{\tau_0}{1+\tau_0}}(p,q)=s$, we have the asymptotic equality
\begin{equation}\label{eq:inv18}
\sum_{k=1}^\infty\frac{\eps^k}{k!}\frac{d^k}{d\tau^k}I_{\frac{\tau}{1+\tau}}(p,q)+\binom{n}{m-1}\sum_{k=0}^\infty\frac{\eps^k}{k!}\frac{d^k}{d\tau^k}\left(e^{-\chi(\tau)}S(\tau)\right)=0,
\end{equation}
where the derivatives are evaluated at $\tau=\tau_0$. 
The construction of the coefficients $\tau_j$ can be done with the help of symbolic calculations. The technical details of the manipulations to find \eqref{eq:inv13} and  \eqref{eq:inv14} are available from the authors.

\begin{example}\label{ex:ex02}
We take $n=99$, $m=49$ and try to find the value $\theta$ such that $S_{100,50}^\prime(\theta)=\frac12$.
This means, we try to find the transition value $\theta_t$ for these $m$ and $n$; see Definition~\ref{def:def01}.
This example corresponds with the second line in Table~\ref{tab:tab02}.  For the asymptotic method we have the following steps.
\begin{enumerate}
\item
Compute the saddle point $z_0\doteq39.1327$ by solving the equation $\phi^\prime(z)=0$, see \eqref{eq:summ12}, and  $t_0=49/50=0.98$ from $\chi^\prime(t)=0$,  see \eqref{eq:summ14}.
\item
Compute $\phi(z_0)\doteq259.198$ and $\chi(t_0)\doteq68.6165$.
\item
With $s=\frac12$ solve the equation $I_x(p,q)=\frac12$, see \eqref{eq:inv02}. We find  
 $x\doteq0.4899330675$.
This gives $\tau_0=x/(1-x)\doteq0.960527$ and $\chi(\tau_0)\doteq68.6215$.
\item
Use \eqref{eq:summ13} to compute $\theta$ from the equation $\phi(\theta)=\phi(z_0)+\chi(\tau_0)-\chi(t_0)\doteq259.203$  in the interval $0,z_0)$ (because $\tau_0<t_0$, and  find
$\theta \doteq 38.29722$. 
\item
A first check: compute $S_{100,50}^\prime(\theta)$ with this value of $\theta$ and  find $0.50233$, with relative error $0.0047$.
\item
Next compute $\tau_1$ from \eqref{eq:inv13}, with the just found value of $\theta$ that is needed in $f_0=f(t_0)$ given in  \eqref{eq:summ18}. Find $\tau_1\doteq-0.055873923$ and compute $\tau\sim \tau_0+\tau_1/\nu\doteq0.959409535$, with $\nu=n-m$.

\item
Repeat the steps given above: the equation for the new $\theta$ becomes $\phi(\theta)=\phi(z_0)+\chi(\tau_0+\tau_1/\nu)-\chi(t_0)\doteq259.203352$ and find $\theta \doteq  38.2492993$.
\item
Check: compute $S_{100,50}^\prime(\theta)$ with this $\theta$ and  find $0.5000190$, with relative error 0.38$e-$4.
\item
With   the next term in the expansion,  $\tau\sim \tau_0+\tau_1/\nu+\tau_2/\nu^2$,  we find $\theta\doteq38.248908191$, and $S_{100,50}^\prime(\theta)\doteq0.500000125$, with relative error 0.25$e-$6.
\end{enumerate}
\end{example}

\renewcommand{\arraystretch}{1.2}
\begin{table}
\caption{\small
Results for inverting  equation $S_{n,m}^\prime(\theta)=s$ by using Newton iteration with starting value  $\theta=\theta_0$ and  relative errors $\delta_j=\vert s/S_{n,m}^\prime(\theta_j)-1\vert$, $j=0,2,4$.
\label{tab:tab01}}
$$
\begin{array}{rccccccc}
m/n\quad & s &\theta_0&\delta _0& \theta_2&\delta_2 & \theta_4 & \delta_4\\
\hline
10/25        \quad &  0.0001 &  0.02 & 0.82 & 0.812 & 0.20 e$-$00 & 0.78467 & 0.17 e$-$03\\
10/25        \quad &  0.25 &       4.55 & 0.36 & 3.786 & 0.46 e$-$07 & 3.78618 & 0.00e$-$00\\
10/25        \quad &  0.50 &      6.13 & 0.23 & 5.163 & 0.46 e$-$03 & 5.16527 & 0.10 e$-$14\\
10/25        \quad &  0.75 &      8.21 & 0.12 & 6.970 & 0.26 e$-$02 & 6.98945 & 0.98 e$-$10\\
25/50        \quad &  0.0001 &    6.20 & 0.64 & 5.70 & 0.46 e$-$01 & 5.67813 & 0.40 e$-$06\\
25/50        \quad &  0.25 &   16.06 & 0.24 & 14.941 & 0.91 e$-$05 & 14.9416 & 0.10 e$-$14\\
25/50        \quad &  0.50 &    19.70 & 0.15 & 18.373 & 0.18 e$-$04 & 18.3727 & 0.14e$-$14\\
25/50        \quad &  0.75 &    24.14 & 0.080  & 22.563 & 0.19 e$-$03 & 22.5663 & 0.10 e$-$14\\
\hline
\end{array}
$$
\end{table}
\renewcommand{\arraystretch}{1.0}

\subsection{Numerical results for the inversion}\label{sec:numres}

In Table~\ref{tab:tab01} we give the results for computing $\theta$ from the equation $S_{n,m}^\prime(\theta)=s$ by using Newton iteration (see \S\ref{sec:invmeth1}).  The starting value $\theta_0$ is obtained  by inverting the reduced equation in \eqref{eq:inv02}, where also the relation between $x$ and $\theta$ is explained. In the table we give the iterated values of  $\theta_j$ and $\delta_j=\vert s/S_{n,m}^\prime(\theta_j)-1\vert$ for $j=0,2,4$.
As can be expected by using Newton iteration, once we have a reasonable starting value, we can obtain excellent accuracy with a few iteration steps. We even see convergence for small values $s=0.0001$, where for the corresponding $\theta$ values the curve of 
 $S_{n,m}^\prime(\theta)$ is very flat (see Figure~\ref{fig:fig01}, {\bf Right}).

In Table~\ref{tab:tab02} we give the results for computing $\theta$ by using asymptotic methods described in \S\ref{sec:invmeth2}.  The equation $S_{n+1,m+1}^\prime(\theta)=s$ is approximately solved by using the approximation of $\tau$  in   \eqref{eq:inv12} with terms up to $\tau_j$, $j=0,1,2$.  As expected, we see  a better performance when we include  $\tau_1/\nu$ and $\tau_1/\nu+\tau_2/\nu^2$.

\renewcommand{\arraystretch}{1.2}
\begin{table}
\caption{\small
Results for computing $\theta$ from the equation $S_{nm}^\prime(\theta)=s$ by using the approximation of $\tau$  in   \eqref{eq:inv12} with terms up to $\tau_j$, $j=0,1,2$; see \eqref{eq:inv13}  and \eqref{eq:inv14}. We give the corresponding values $\theta_j$ and  relative errors $\delta_j=\vert s/S_{nm}^\prime(\theta)-1\vert$.
\label{tab:tab02}}
$$
\begin{array}{rccccccc}
m/n\quad & s &\theta_0&\delta _0& \theta_1&\delta_1 & \theta_2 & \delta_2\\
\hline
200/250      \quad &  0.0001 &    255.3 &  0.36e$-$2 &  255.339 &  0.24e$-$4 &  255.33835 &  0.13e$-$4\\
200/250      \quad &  0.25 &    408.2 &  0.11e$-$2 &  408.103 &  0.35e$-$5 &  408.10264 &  0.66e$-$7\\
200/250      \quad &  0.50 &    455.0 &  0.66e$-$3 &  454.911&  0.21e$-$5 &  454.91098 &  0.18e$-$6\\
200/250      \quad &  0.75 &    508.2 &  0.34e$-$3 &  508.124 &  0.11e$-$5 &  508.12328 &    0.76e$-$8\\
500/1000        \quad &  0.0001 &    307.4 &  0.73e$-$2 &  307.383 &  0.58e$-$5 &  307.38266 &  0.32e$-$8\\
500/1000        \quad &  0.25 &       378.6 &  0.23e$-$2 &  378.570 &  0.19e$-$5 &  378.56980 &  0.10e$-$8\\
500/1000        \quad &  0.50 &       396.4 &  0.15e$-$2 &  396.387 &  0.11e$-$5 &  396.39298 &  0.20e$-$3\\\
500/1000        \quad &  0.75 &      415.1 &  0.77e$-$3 &  415.025 &  0.62e$-$6 &  415.02539 &  0.20e$-$9\\
\hline
\end{array}
$$
\end{table}
\renewcommand{\arraystretch}{1.0}

\section{Deriving the complete asymptotic expansion}\label{sec:deriv}
The first-term approximation in our previous result in \eqref{eq:summ16} of Theorem~\ref{thm:stel04} will now be extended in the following theorem.
\begin{theorem}\label{thm:stel06}  Let  $R_{n+1,m+1}^\prime(\theta)$ be defined as in \eqref{eq:summ16}. Then for $N=0,1,2,\ldots$
\begin{equation}\label{eq:deriv01}
R_{n+1,m+1}^\prime(\theta)=
e^{-\chi(\tau)}\binom{n}{m-1}\sum_{k=0}^{N-1}\frac{G_k(t_0)}{\nu^k}+\frac{e^{-\chi(\tau)}}{2\pi i\nu^N}\int_{{\cal C}_\sigma} \frac{(t+1)^n}{t^{m}}G_N(t)\,dt,
\end{equation}
where $\nu=n-m$,  $G_0(t)=g(t)$, see \eqref{eq:summ15}, and other  $G_k(t)$ follow from the recursive scheme
\begin{equation}\label{eq:deriv02}
H_k(t)=\frac{G_k(t)-G_k(t_0)}{t-t_0},\quad 
G_{k+1}(t)=-\frac{d}{dt}\Bigl(t(1+t)H_k(t)\Bigr),\quad k=0,1,2,\ldots\,.
\end{equation}
\end{theorem}

\begin{proof}
We start with the representation in  \eqref{eq:summ16} and use an integration by parts procedure, which starts by writing  
\begin{equation}\label{eq:deriv03}
G_0(t)=G_0(t_0)+(t-t_0) H_0(t),\quad G_0(t)=g(t).
\end{equation}
This gives
\begin{equation}\label{eq:deriv04}
R_{n+1,m+1}^\prime(\theta)=
G_0(t_0)e^{-\chi(\tau)}\binom{n}{m-1}+\frac{e^{-\chi(\tau)}}{2\pi i}\int_{{\cal C}_\sigma} \frac{(t-t_0)H_0(t)}{\chi^\prime(t)}\,d  e^{\chi(t)},
\end{equation}
and using $\chi^\prime(t)$ shown  in \eqref{eq:summ14}, we find
\begin{equation}\label{eq:deriv05}
R_{n+1,m+1}^\prime(\theta)=
G_0(t_0)e^{-\chi(\tau)}\binom{n}{m-1}+\frac{e^{-\chi(\tau)}}{2\pi i\nu}\int_{{\cal C}_\sigma} \frac{(t+1)^n}{t^{m}}G_1(t)\,dt,
\end{equation}
where
\begin{equation}\label{eq:deriv06}
G_1(t)=-\frac{d}{dt}\bigl(t(1+t)H_0(t)\bigr).
\end{equation}
This integral has the same form as the one in \eqref{eq:summ16}, and we can continue this method. This proves the theorem.

\end{proof}

 The first coefficients  $G_k(t_0)$ of the expansion in \eqref{eq:deriv01} are
\begin{equation}\label{eq:deriv07}
\begin{array}{r@{\,}c@{\,}l}
G_0(t_0)&=& g_0,\quad G_1(t_0)= -(1+2t_0)g_1-t_0(t_0+1)g_2,\\[8pt]
G_2(t_0)&=& 2(1+2t_0)g_1+(2+11t_0+11t_0^2)g_2\ +\\[8pt]
&&5t_0(t_0+1)(1+2t_0)g_3+3t_0^2(t_0+1)^2g_4.
\end{array}
\end{equation}
The coefficients $g_k$ follow from the coefficients $f_k$ in the expansion 
$\dsp{f(t)=\sum_{k=0}^\infty f_k(t-t_0)^k}$,
with $f(t)$ defined in \eqref{eq:summ14}. We have
\begin{equation}\label{eq:deriv08}
g(t)=f(t)-\frac{1}{t-\tau}=\sum_{k=0}^\infty g_k(t-t_0)^k\quad g_k=f_k -\frac{(-1)^k}{(t_0-\tau)^{k+1}}.
\end{equation}

Finally, all these coefficients can be expressed in terms of the coefficients $z_k$  of the expansion
\begin{equation}\label{eq:deriv09}
z-z_0=\sum_{k=1}^\infty z_k (t-t_0)^k,
\end{equation}
and these follow from substituting this expansion in the Taylor expansions of the functions used in the transformation given in \eqref{eq:summ11}. This transformation has the local expansions at the saddle points
\begin{equation}\label{eq:deriv10}
(z-z_0)\sqrt{\sum_{k=2}^\infty \frac{1}{k!} \phi^{(k)}(z_0)(z-z_0)^{k-2}}=
(t-t_0)\sqrt{\sum_{k=2}^\infty \frac{1}{k!} \chi^{(k)}(t_0)(t-t_0)^{k-2}},
\end{equation}
where the square roots are positive for positive values of $z$ and $t$.
The derivatives $ \phi^{(k)}(z)$ can be expressed in terms of the derivatives of the gamma functions; see \eqref{eq:summ12}. The derivatives $\chi^{(k)}(t)$ at $t=t_0$ are simple expressions.

The first coefficients $z_k$ are
\begin{equation}\label{eq:deriv11}
z_1 =\sqrt{\frac{\chi^{(2)}(t_0)}{\phi^{(2)}(z_0)}}, \quad 
z_2= \frac{\chi^{(3)}(t_0)-z_1^3\phi^{(3)}(z_0)}{6z_1\phi^{(2)}(z_0)}.
\end{equation}

The first coefficients $f_k$ are
\begin{equation}\label{eq:deriv12}
\begin{array}{r@{\,}c@{\,}l}
f_0&=& \dsp{\frac{z_1}{z_0-\theta}},\quad f_1=\dsp{\frac{2z_2z_0-2z_2\theta-z_1^2}{(\theta-z_0)^2}}\,\\[8pt]
f_2&=& \dsp{\frac{6z_3z_0\theta+3z_1z_2z_0-3z_1z_2\theta-z_1^3-3z_3z_0^2-3z_3\theta^2}{(\theta-z_0)^3}}.
\end{array}
\end{equation}
Then, the first coefficients $g_k$ follow from \eqref{eq:deriv08}.

 \renewcommand{\arraystretch}{1.2}
\begin{table}
\caption{\small
Values of the relation \eqref{eq:deriv13} in the computation of  $S^\prime_{n,m}(\theta)$  for $n=1.000$ and $n=100.000$ for several values of $m$ and $\theta$.
\label{tab:tab03}}
$$
\begin{array}{ccccccc}
\hline
n=1.000 &{\rm Digits}=16&4 \ {\rm terms}&&&&\\
\hline
\rho \ \backslash \  m & 150 &300&450& 600&750 & 900\\
\hline
0.70 & 0.16e$-$12 & 0.27e$-$13 & 0.73e$-$13 & 0.10e$-$12 & 0.12e$-$12 & 0.47e$-$13\\
0.80 & 0.32e$-$12 & 0.13e$-$12 & 0.16e$-$12 & 0.86e$-$13 & 0.11e$-$12 & 0.88e$-$13\\
0.90 & 0.24e$-$11 & 0.30e$-$12 & 0.70e$-$13 & 0.60e$-$13 & 0.50e$-$12 & 0.13e$-$10\\
1.00 & 0.12e$-$11 & 0.19e$-$11 & 0.14e$-$12 & 0.54e$-$12 & 0.29e$-$11 & 0.79e$-$13\\
\hline
n=1.000 &{\rm Digits}=20&4 \ {\rm terms}&&&&\\
\hline
\rho \ \backslash \  m & 150 &300&450& 600&750 & 900\\
\hline
0.70 & 0.78e$-$15 & 0.58e$-$16 & 0.17e$-$17 & 0.49e$-$16 & 0.16e$-$16 & 0.24e$-$16\\
0.80 & 0.47e$-$15 & 0.48e$-$16 & 0.93e$-$17 & 0.16e$-$16 & 0.60e$-$18 & 0.20e$-$16\\
0.90 & 0.12e$-$15 & 0.36e$-$16 & 0.28e$-$16 & 0.60e$-$17 & 0.56e$-$16 & 0.20e$-$14\\
1.00 & 0.17e$-$15 & 0.14e$-$15 & 0.93e$-$16 & 0.37e$-$16 & 0.37e$-$16 & 0.95e$-$16\\
\hline
n=100.000 &{\rm Digits}=16&2 \ {\rm terms}&&&&\\
\hline
\rho \ \backslash \  m & 15.000 & 30.000& 45.000&  60.000& 75.000 &  90.000\\
\hline
0.97 &  0.53e$-$10 &  0.36e$-$10 &  0.86e$-$11 &  0.82e$-$11 &  0.24e$-$11 &  0.47e$-$11\\
0.98 &  0.58e$-$11 &  0.12e$-$10 &  0.61e$-$11 &  0.14e$-$10 &  0.16e$-$11 &  0.17e$-$10\\
0.99 &  0.86e$-$11 &  0.17e$-$10 &  0.18e$-$10 &  0.16e$-$10 &  0.20e$-$10 &  0.44e$-$11\\
1.00 &  0.34e$-$09 &  0.28e$-$08 &  0.47e$-$09 &  0.14e$-$08 &  0.46e$-$11 &  0.19e$-$08\\
\hline
n=100.000 &{\rm Digits}=20&2 \ {\rm terms}&&&&\\
\hline
\rho \ \backslash \  m & 15.000 & 30.000& 45.000&  60.000& 75.000 &  90.000\\
\hline
0.97 &  0.20e$-$14 &  0.21e$-$14 &  0.16e$-$14 &  0.16e$-$14 &  0.63e$-$15 &  0.97e$-$15\\
0.98 &  0.37e$-$14 &  0.38e$-$14 &  0.15e$-$14 &  0.29e$-$16 &  0.80e$-$15 &  0.25e$-$15\\
0.99 &  0.85e$-$14 &  0.10e$-$14 &  0.18e$-$14 &  0.14e$-$14 &  0.16e$-$14 &  0.64e$-$15\\
1.00 &  0.71e$-$13 &  0.22e$-$12 &  0.23e$-$13 &  0.15e$-$13 &  0.18e$-$12 &  0.51e$-$13\\
\hline
\end{array}
$$
\end{table}
\renewcommand{\arraystretch}{1.0}

 \subsection{Numerical verifications of the asymptotic approximation}\label{sec:num}
A convenient tool for verifying the errors in numerical calculations is the recursion in Theorem~\ref{thm:stel01},  \eqref{eq:details04}. We can write this in the form
 \begin{equation}\label{eq:deriv13}
\frac{nS^\prime_{n,m}(\theta)+\theta S^\prime_{n,m-1}(\theta)}{(\theta+n)S^\prime_{n+1,m}(\theta)}-1=0.
\end{equation}
Especially for large values of $n$ (for example, $n=100.000$) used in the tests we avoid the exact evaluation of the sums in \eqref{eq:intro01}, and we accept that we do not verify a standard relative error.

In Table~\ref{tab:tab03}  we show the values of the relation \eqref{eq:deriv13} in the computation of  $S^\prime_{n,m}(\theta)$  for $n=1000$ and $n=100.000$ for several values of $m$. In the first two parts ($n=1.000$) of the table we used the expansion in \eqref{eq:deriv01} with terms up to and including $k=3$, and in the final two parts ($n=100.000$) we only used the terms with $G_0(t_0)$ and $G_1(t_0)$.  We have taken $\theta=\rho z_0$, for the shown values of $\rho$. In this way the values of $S^\prime_{n,m}(\theta)$ are not very small. For example, with $n=1000.000, \ m=75.000, \  z_0\doteq136312.21$, we have
 \begin{equation}\label{eq:deriv14}
\begin{array}{ll}
\rho=0.97 \quad \Longrightarrow \quad  S^\prime_{n,m}(\theta) \doteq 0.300778124649e$-$04,\\[8pt]
\rho=1.00  \quad \Longrightarrow \quad  S^\prime_{n,m}(\theta) \doteq 0.501722781430e$-$00.
\end{array}
\end{equation}

The computations are done with  Maple 2020, Digits=16. To show the effect of cancellation or rounding errors, we have used the same Maple codes with Digits=20. 

The first three coefficients $G_k(t_0)$ of the expansion in \eqref{eq:deriv01} are given in \eqref{eq:deriv07}. Each $G_k(t_0)$ is a linear combination of coefficients $g_k$, $k=1,2,\ldots, 2k$, and these are defined in \eqref{eq:deriv08}. When $\theta\sim\theta_t$, the transition value, $\vert t_0-\tau\vert$ is small, and in the limit $\tau \to t_0$ the coefficient $g_k$ is well defined, although  $f_k$ has a pole at  $t=t_0$ with corresponding $z$-value $\theta=z_0$ (see the definition of $f(t)$ in \eqref{eq:summ14}). In  $g_k$ these poles are cancelled. From an analytical point of view, everything runs fine, but the algorithm needs special attention. For this we use expansions of $g_k$ for small values of $\vert t_0-\tau\vert$ in the form
\begin{equation}\label{eq:deriv15}
g_k=\sum_{j=0}^\infty g_{j,k}(\tau-t_0)^j,\quad k=0,1,2,\ldots\,. 
\end{equation}

The largest errors in the first part and third part (both with Digits=16) in Table~\ref{tab:tab03} occur for $\rho=1$, that is, when $\theta=z_0$, hence $\theta\sim\theta_t$, the transition value. In this case $\tau\sim t_0$. This is related to  the  difficulty of computing the coefficients $G_k(t_0)$ near the transition point $\theta_t$, which is near $z_0$, as explained in Definition~\ref{def:def01} ad Remark~\ref{rem:remark01}. This is an interesting analytical issue, but for the population genetics application area it is less important.  In fact we accept some loss of accuracy near the transition point instead of using more complicated analytical expansions of the coefficients.

\section{Conclusions}
We have provided additional details for the asymptotic approximation of the cumulative distribution quantities $S^\prime_{n,m}(\theta)$ and $T^\prime_{n,m}(\theta)=1-S^\prime_{n,m}(\theta)$. As a completely new contribution, we have considered the inversion problem to compute  $\theta$, the solution of the equation $S^\prime_{n,m}(\theta)=s$, with given $s\in(0,1)$, $m$ and large $n$. A simple inversion method uses Newton iteration, the other one asymptotic approximations. For this we provided additional coefficients in the earlier given asymptotic expansion. We have shown that some loss of accuracy is localized near the transition values. We further observe that these estimation errors can be mitigated with additional terms of the expansion. These errors are largely in a regime where the distribution functions are near the changeover value $\frac12$,  which is an interesting domain from an analytical point of view, but where Fu's $F_s$ values are not important for genetic inferences.

\section*{Acknowledgments}
SLC acknowledges members of the Chen lab for fruitful discussions. NMT acknowledges the help of Edgardo Cheb-Terrab (Maplesoft) and CWI, Amsterdam for scientific support.\\
SLC was supported by the National Medical Research Council, Ministry of Health, Singapore (grant numbers NMRC/OFIRG/0009/2016 and\\ NMRC/CIRG/1467/2017).\\
NMT was supported by the Ministerio de Ciencia e Innovaci\'on, Spain, projects MTM2015-67142-P (MINECO/FEDER, UE) and \\ 
PGC2018-098279-B-I00 (MCIU/AEI/FEDER, UE).


\end{document}